\documentclass{amsart}
\usepackage[margin=1in]{geometry}
\usepackage{amsmath}
\usepackage{amssymb}
\usepackage{amsthm}
\usepackage{mathrsfs}
\usepackage{enumerate}
\usepackage{tikz}
\newtheorem{thm}{Theorem}[section]
\newtheorem{lem}[thm]{Lemma}
\newtheorem{cor}[thm]{Corollary}
\theoremstyle{definition}
\newtheorem{dfn}[thm]{Definition}
\newtheorem{dfns}[thm]{Definitions}

\newtheorem*{prb*}{Problem}
\newtheorem{ex}[thm]{Example}
\newtheorem{rmk}[thm]{Remark}
\newtheorem{conjecture}[thm]{Conjecture}

\newcommand{\newP}{M} 
\newcommand{\newF}{P} 

\usepackage{hyperref}

\usepackage{tabularx} 

\providecommand{\emp}{\varnothing}

\providecommand{\N}{\mathbb{N}}

\providecommand{\Q}{\mathbb{Q}}
\providecommand{\R}{\mathbb{R}}

\title[Point configurations in sets of sufficient topological structure]{\centering{Point configurations in sets of sufficient topological structure and a topological {E}rd\H{o}s similarity conjecture}}

\author{Alex McDonald}
\address{Department of Mathematics, Kennesaw State University, Marietta, GA}
\email{amcdon79@kennesaw.edu}

\author{Krystal Taylor}
\address{Department of Mathematics, The Ohio State University, Columbus, OH}
\email{taylor.2952@osu.edu}

\date{}

\begin{document}

\maketitle

\begin{abstract}

We explore the occurrence of point configurations in non-meager
Baire sets. 
A celebrated result of Steinhaus asserts that $A+B$ and $A-B$ contain an interval whenever $A$ and $B$ are sets of positive Lebesgue measure in $\R^d$ for $d\geq 1$. 
A topological analogue attributed to Piccard asserts that both $AB$ and $AB^{-1}$ contain an interval when $A,B$ are non-meager Baire sets in a topological group.  
We explore
generalizations of Piccard's result to more complex point configurations and more abstract spaces. 
In the Euclidean setting, we show that if $A\subset \R^d$ is a non-meager Baire set and $\newF=\{v^i\}_{i\in\N}$ is a bounded sequence, then there is an interval of scalings $t$ for which $t\newF+z\subset A$ for some $z\in \R^d$. 
That is, the set 
$$\Delta_\newF(A)=\{t>0: \exists z\text{ such that }t\newF+z\subset A\}$$
has nonempty interior. 
More generally, if $V$ is a topological vector space and $\newF=\{v^i\}_{i\in\N}\subset V$ is a bounded sequence, we show that if $A\subset V$ is non-meager and Baire, then $\Delta_\newF(A)$ has nonempty interior. 
The notion of boundedness in this context is described below.  Note that the sequence $\newF$ can be countably infinite, which distinguishes this result from its measure-theoretic analogue. In the context of the   topological version of Erd\H{o}s' similarity conjecture, we show that bounded countable sets are universal in non-meager Baire sets. 
\end{abstract}

\section{Introduction}
A large body of research concerns the occurrence (or absence) of patterns in sets of sufficient Hausdorff dimension, see for instance, \cite{GM22,GIT, IosMag, Areas, OT}, or in sets of sufficient Fourier dimension \cite{FraserPram, CLP, Shmerkin}. Another notion of size that has proved useful in the study of patterns is that of Newhouse thickness, see for instance, \cite{McT1, McT2, YSurvey, SandbergTaylor}. 
In this article, we focus on the occurrence of \textit{infinite} point configurations within sets of sufficient topological structure. 
\vskip.12in

Our story begins with a well-known theorem of Steinhaus from the 1920s and its topological analogue, known as Piccard's theorem. 
Let $A, B\subset \R^d$ denote sets of positive measure. 
Steinhaus' theorem states that
the set $A-B$ contains a nontrivial open set (see \cite{Steinhaus} for $d=1$ and \cite{Keperman} for $d\geq 1$).
Said differently, there is an open set of shifts $t$ so that 
\begin{equation}\label{intersect2}
A\cap (B+t) \neq \emp.
\end{equation}
This assertion can be proved by observing that the convolution of the characteristic functions $\chi_A$ and $\chi_B$ is a continuous, so that the function
$$t\rightarrow \int\chi_B(t-y) \chi_A(y) dy$$
is continuous and, by Fubini, not identically zero. 
Alternatively, \eqref{intersect2} follows as a simple consequence of the Lebesgue density theorem. 
\vskip.12in

For more intricate configurations, these same proof techniques can be recycled to show that for each $k\in\N$, there exists an open ball $S_k$ about the origin so that if $\{v^i\}_{i=1}^k \subset S_k$, then 
\begin{equation}\label{intersectN}
\bigcap_{i=1}^k (A+v^i)\neq \emp.
\end{equation}
It is an immediate consequence that, if $\newF=\{v^i\}_{i=1}^k$ denotes an arbitrary finite point set, then $A$ contains all sufficiently small affine copies of $\newF$. 
That is, if
$$\Delta_\newF(A)=\{t>0: \exists z\in \R^d\ t\newF+z\subset A\},$$
then $\Delta_\newF(A)$ contains an interval with left endpoint $0$. 
\vskip.12in

There are many further generalizations and modifications of Steinhaus' theorem, including a result of Erd\H{o}s and Oxtoby showing that if $F$ on $\R\times \R$ is a continuously differentiable real-valued function with non-vanishing partial derivatives on its domain, then $F(A,B)$ has nonempty interior \cite{ErdosOxtoby}. This theorem also holds in higher dimensions, see \cite[Chapter II]{Jarai} and the references therein and \cite[Proposition 2.9 (i)]{STinterior}. 
\vskip.12in

A topological analogue of Steinhaus' theorem known as Piccard's theorem asserts that the sum of two non-meager Baire sets (see Definition \ref{Baire}) in a topological vector space contains a nonempty open set \cite{Piccard, Kominek1971}. 
Again, addition can be replaced by a two variable function satisfying minimal differentiability assumptions; see  \cite[Proposition 2.9 (iii)]{STinterior} and \cite[Chapter II]{Jarai}. 
There is also an analogous statement for topological groups, which states that $A*B$ (and $A*B^{-1}$) contains a nonempty open set provided that $A$ and $B$ are non-meager Baire sets \cite{bhaskara}. 
\vskip.12in

While the aforementioned results consider operations on pairs of points, the present article focuses on general patterns and universality. 
A set or pattern $\newF$ is universal in the collection of sets $\mathcal{X}$ if for all $A\in \mathcal{X}$, we can find an affine copy of $\newF$ inside of $A$. 
\vskip.12in

Universality can be considered for different classes of sets. 
We call a set \textbf{measure universal} if it is universal in the collection of Lebesgue measurable sets. 
Finite sets are known to be measure universal by Steinhaus' theorem.
Further, we say that $P$ is \textbf{full measure universal} if it is universal in the collection of measurable sets $A$ with the property that the compliment of $A$ has measure zero. 
Measure universality implies full measure universality. 
However, while it is easy to see that countably infinite sets are full measure universal (see Example \ref{baire category ex}), it is an open question whether there exists such a set that is measure universal. 
\vskip.12in

The classic Erd\H{o}s similarity conjecture asserts that no infinite sets can be measure universal. 
While there are several examples of bounded countably infinite sets that are not measure universal, it is an open question whether there exists a countably infinite sequence that is measure universal.  
 For a summary of results 
 on recent variants of this conjecture of Erd\H{o}s, see the survey of Jung, Lai, and Mooroogen \cite{JLM}. 
\vskip.12in

Erd\H{o}s' similarity conjecture can also be studied in a topological setting. 
Here, Baire sets are the analogue of Lebesgue measurable sets,  non-meager sets are the analogue of positive measure sets, and comeager sets are the analogue of full measure sets (see Definition \ref{Baire} and Table \ref{table}). 
\vskip.12in

Universality for the collection of comeager sets is completely  resolved, which was recently considered by Gallagher, Lai and Weber \cite{GLW}.  
Indeed, a set is universal among comeager sets if and only if it is a strong measure zero set \cite{Galvin} (also see \cite{JL}). 
Since Borel's conjecture, which states that all strong measure zero sets are countable, 
is known to be independent of ZFC \cite{sierpinski}, it follows that the existence of an uncountable set that is universal among comeager sets is independent of ZFC. 
%
 %

\vskip.12in

In this article we provide the missing link for the topological variant of the classic Erd\H{o}s similarity conjecture.  
We show that all bounded countable sets are universal in the collection of non-meager Baire sets.
Other than results inferred from the obvious implication that lack of comeager universality implies lack of non-meager universality, no previous results on non-meager universality have appeared in the literature to the authors' knowledge. 
\vskip.12in


To put this into further context, it is a straightforward consequence of the Baire category theorem that countably infinite sets are universal in the collection of comeager sets (see Example \ref{baire category ex}). 
 Our main Euclidean result (Theorem \ref{euclidean}) extends this observation to the collection of non-meager Baire sets. 
Further, we prove a more quantitative result about an interval worth of scalings, and our results extend to topological vector spaces (Theorem \ref{TVS}). 
\vskip.12in

To end, we pose the following direct topological analogue of the Erd\H{o}s similarity conjecture: 
\begin{conjecture}
    For each uncountable set $\newF$, there exists a non-meager Baire set $A$ that contains no affine image of $\newF$. 
\end{conjecture}
\vskip.12in

\subsection{Background and definitions}
 Before stating our main results, we give some definitions and background results. 
\vskip.12in

\begin{dfns}\label{Baire}
Let $X$ denote a topological space.  A subset of $X$ is \textbf{meager} (first category) if it is a countable union of nowhere dense sets.  A set is \textbf{non-meager} (second category) if it is not meager, and it is \textbf{comeager} if its complement is meager. 
Finally, a set $A\subset X$ is said to be \textbf{Baire} if there is an open set $U$ and a meager set $\newP$ so that $A$ equals the symmetric difference $A= U\Delta \newP$.  
\end{dfns}

An example of a set that is non-meager but not comeager is $[0,1]$; more generally, any set with both nonempty interior and exterior is a non-meager set which is not comeager.

\begin{rmk}
Equivalently, one may define Baire sets as sets $A$ for which there exists an open set $U$ such that $U\Delta A$ is meager.  If $A$ satisfies this condition, then $A$ satisfies Definition \ref{Baire} with $P=U\Delta A$.  
Conversely, if $U$ satisfies Definition \ref{Baire}, then $A\Delta U=\newP$ is meager.  See Theorem \ref{equivalentdefinitions} for other useful equivalent definitions.
\end{rmk}

\bigskip
\begin{table}[h]\label{table}
\begin{tabularx}{0.8\textwidth} { 
  | >{\raggedright\arraybackslash}X 
  | >{\centering\arraybackslash}X 
  | >{\raggedleft\arraybackslash}X | }
 \hline
 \textbf{Measure notions} & \textbf{Topological notions} \\ 
 \hline
sets of full measure & comeager sets   \\
 \hline
sets of positive measure  & non-meager sets   \\
 \hline
measurable sets  & Baire sets   \\
 \hline
measure zero sets & meager sets\\
\hline
\end{tabularx}
\end{table}
\bigskip

Morally, one may think of meager sets as the topological analogues of measure zero sets (in the sense that they are considered negligible and are stable under countable unions), and of Baire sets as the topological analogues of measurable sets (in the sense that they are ``almost'' equal to open sets up to negligible error). A word of caution: even meager sets can have full measure. 
A succinct account of basic facts and examples is given in Appendix \ref{Appendix}; for a more thorough introduction, see \cite[Chapter I.8]{Kechris}. 
\vskip.12in

\begin{thm}[Piccard - Euclidean setting]
\label{picard}
If $A,B\subset \R^d$ are non-meager Baire sets, then both $A+B$ and $A-B$ have nonempty interior.
\end{thm}
A short proof of Theorem \ref{picard} is given in Section \ref{Euclidean} for emphasis.  However, we also give a proof of the following more general result in Section \ref{TVSsection}. 
\vskip.12in

\begin{thm}[Piccard - Abstract version]
\label{picard_gen}
Let $X$ denote a topological vector space (or a topological group).
If $A,B\subset X$ are non-meager Baire sets, then both $A+B$ and $A-B$ (both $A*B$ and $A*B^{-1}$) contain nonempty open sets. 
\end{thm}
\vskip.12in

The definition of $\Delta_\newF(A)$ given above immediately extends to vector spaces:
\vskip.12in

\begin{dfn}
Let $V$ be a real vector space.  Given sets $A,\newF\subset V$, define 
\[
\Delta_\newF(A)=\{t\in\R: \exists z\in V\ t\newF+z\subset A\}.
\]
\end{dfn}
\vskip.12in

In the notation above, a consequence of Piccard's Theorem is that $\Delta_\newF(A)$ has nonempty interior whenever $A$ is a non-meager Baire set and $\newF$ is a two-point set.  In the following sections, we explore generalizations of this result to more complex point configurations and more abstract spaces. To make sense of the set $\Delta_\newF(A)$, we need both translations and scalings, and for this reason we work in topological vector spaces. 
\vskip.12in

\subsection{Acknowledgment}
K.T. is supported in part by the Simons Foundation Grant GR137264. The authors thank Angel Cruz, Yeonwook Jung, Yuveshen Mooroogen, and the anonymous referee for interesting discussions and suggestions related to this article. 

\section{Results in Euclidean space}\label{Euclidean}
Our first main result is the following generalization of Piccard's Theorem.
\begin{thm}
\label{euclidean}
Let $A\subset \R^d$ be a non-meager Baire set, and let $\newF=\{v^i\}_{i\in\N}$ be a bounded sequence.  Then, $\Delta_\newF(A)$ has nonempty interior.
\end{thm}
The necessity of the hypotheses are exemplified in Section \ref{Appendix}.
The proof of Theorem \ref{euclidean} is a topological variant of the classic proof of Steinhaus's Theorem.  The main idea of Steinhaus's Theorem via the Lebesgue density theorem is as follows.  If $A\subset \R^d$ is a measurable set, then a ball $B$ can be chosen so that
$A$ takes up $99\%$ of the space of $B$.
This means $A$ is guaranteed to intersect small translations of itself, which is the crux of the proof of the theorem. 
This approach works for finite configurations; as there is finite loss of measure with each intersection. 
If $A$ is instead a non-meager Baire set, we may still approximate $A$ by a ball in the following sense. The following is immediate from the definition of a Baire set.

\begin{lem}[Topological variant of Lebesgue Density Theorem]
\label{topLDT}
Let $A\subset \R^d$ be a non-meager Baire set.  There exists an open ball $B$ such that $B\setminus A$ is meager.
\end{lem}
\begin{proof}
Since $A$ is Baire, there is an open set $U$ such that $A\Delta U$ is meager (since $A=U\Delta \newP$ implies $A\Delta U\subset \newP$).  Since $A$ is non-meager, $U$ must be nonempty.  Let $B$ be any open ball contained in $U$.  Since
\[
(B\setminus A)\subset (U\setminus A)\subset A\Delta U,
\]
the set $B\setminus A$ is meager.
\end{proof}
Recall that in the Baire Category setting, we think of meager sets as the analogues of measure-zero sets in measure theory; that is, this is the relevant notion of a set being negligible, and this notion is stable under countable unions.  However, the Lebesgue Density Theorem only guarantees that a ball can approximate a measurable set with (say) a $1\%$ error.  On the other hand, Lemma \ref{topLDT} says that a ball may approximate $A$ with meager error.  This stronger approximation lemma is the reason we can find countable configurations, rather than just finite ones.

Before we prove our main theorem, we give a short proof of Piccard's Theorem  using Lemma \ref{topLDT} in the case $d=1$ (the proof when $d\geq 2$ is a simple exercise). Theorem \ref{picard} of course follows from Theorem \ref{picard_gen}, but we give the proof here anyways explicitly in the Euclidean setting. 

\begin{proof}[Proof of Theorem \ref{picard} (one dimensional case)]
For $t\in \R$, define $B_t:=t-B$.  We have $t\in A+B$ if and only if $A\cap B_t \neq \emp$.  Apply Lemma \ref{topLDT} to get non-degenerate intervals $I,J$ such that $I\setminus A$ and $J\setminus B$ are meager, and let $J_t=t-J$.  For any $t$, the set $J_t\setminus B_t$ must be meager also.  Therefore, the set
\[
(I\cap J_t)\setminus (A\cap B_t)\subset (I\setminus A)\cup (J_t \setminus B_t)
\]
is meager.  If $A\cap B_t$ were meager, it would follow that $(I\cap J_t)$ is meager.  This cannot happen if $t$ is chosen so that $I\cap J_t$ is a non-degenerate interval. 
Denote the set of such $t$ by $K= \{t: \left(I\cap J_t\right)^{\circ}\neq \emptyset\},$ and observe that $K$ is a non-degenerate interval.  
Therefore, for any $t\in K$, we have $A\cap B_t\neq \emp$. 
\end{proof}
This proof idea generalizes to arbitrary countable sets without too much difficulty.  The key geometric fact we need to use is that the intersection of infinitely many balls still has nonempty interior, provided that each ball is a small translate of a single original ball.

\begin{proof}[Proof of Theorem \ref{euclidean}]
Define $A_{i,t}=A-tv^i$.  We have $\{tv^i+z:i\in\N\}\subset A$ if and only if $z\in \bigcap_{i=1}^\infty A_{i,t}$.  We must show this intersection is nonempty for an interval worth of $t$.  Apply Lemma \ref{topLDT} to obtain an open ball $B$ with $B\setminus A$ meager, and let $B_{i,t}=B-tv^i$.  It follows that $B_{i,t}\setminus A_{i,t}$ is meager for every $i,t$.  Finally, let 
\[
\widetilde{B}_t=\bigcap_{i=1}^\infty B_{i,t}.
\]
The set $\widetilde{B}_t\setminus A_{i,t}$ is meager, since $\widetilde{B}_t\subset B_{n,i}$.  Therefore,
\[
\widetilde{B}_t\setminus \left(\bigcap_{i=1}^\infty A_{i,t}\right)=\bigcup_{i=1}^\infty (\widetilde{B}_t\setminus A_{i,t})
\]
is meager.  If $\bigcap_{i=1}^\infty A_{i,t}$ were meager, then $\widetilde{B}_t$ would be meager.  However, if $t$ is sufficiently small (relative to the radius of $B$ and $\sup_i |v^i|$, which is finite by assumption), then each set $B_{i,t}$ is a small translate of the original ball $B$, and hence the intersection $\widetilde{B}_t$ has nonempty interior.  Therefore, the theorem holds when $J$ is a sufficiently small neighborhood of $0$.
\end{proof}
\section{Results in topological vector spaces}\label{TVSsection}
The purpose of this section is to generalize our result in Euclidean space (Theorem \ref{euclidean}) to the more general setting of topological vector spaces.  We start by making some definitions.  
\begin{dfn}
A \textbf{topological vector space} is a vector space $V$ with a topology such that the operations of addition and scalar multiplication are continuous as functions $V\times V\to V$ and $\R\times V\to V$, respectively.  We will denote the zero vector by a bold $\textbf{0}$, to contrast with the zero scalar $0$.
\end{dfn}

As a warm-up, we give the proof of Theorem \ref{picard_gen}. 
The following proof also works in the case that $X$ is a topological group; one only need change the notation and replace all occurrences of $A-B$ with $A*B^{-1}$. 
\begin{proof}[Proof of Theorem \ref{picard_gen}]
 We establish the claim for $A-B$. 
Let $U_1$, $U_2$ denote open sets, and let $\newP_1$, $\newP_2$ denote meager sets so that $A= U_1\Delta \newP_1 $ and $B = U_2\Delta \newP_2$. 
Let us assume that nonempty open sets are not meager, for otherwise $X$ would also be meager by the Banach category theorem (see Theorem \ref{BCT}) as a union of translates of such a set. 

We observe that 
\begin{equation}\label{key}
    A-B \supset U_1 - U_2,
\end{equation}
which will prove the claim since $U_1-U_2$ is a nonempty open set.
To show \eqref{key}, let $t\in U_1-U_2$. Observe that 
$t\in A-B$ if and only if $\left(t+B\right)\cap A \neq \emptyset.$
Since 
\begin{equation}\label{key2}
     A\cap \left(t+B \right) \supset 
(U_1 \cap (t+U_2))
\backslash (\newP_1 \cup (\newP_2+t)), 
\end{equation} and since $(U_1 \cap (t+U_2))$ is a nonempty open set, while $(\newP_1 \cup (\newP_2+t))$ is meager, it follows that the right-hand-side of \eqref{key2} is nonempty. 

 \end{proof}

Moving on, recall that in Theorem \ref{euclidean}, our hypotheses included that the configuration set under consideration was bounded.  Since an arbitrary topological vector space need not be metric, we need a definition of boundedness in this context.
\begin{dfn}\label{bounded}
    Let $V$ be a topological vector space, and let $\newF\subset V$.  We say $\newF$ is \textbf{bounded} if, for any open set $U\subset V$, there exist $t>0,z\in V$ such that $t\newF+z\subset U$.
\end{dfn}
We are now ready to state the main theorem of this section.
\begin{thm}
\label{TVS}
    Let $V$ be a topological vector space, and let $\newF=\{v^n\}_{n \in \N}\subset V$ be a bounded sequence.  If $A\subset V$ is a non-meager Baire set, then $\Delta_\newF(A)$ has nonempty interior.
\end{thm}

The proof of Theorem \ref{euclidean} depended on several topological properties of the space $\R^d$, and we need analogues of those properties to allow the proof to be carried out in this setting.  First and most obviously, our arguments in the Euclidean setting are based on the fact that $\R^d$ is a Baire space, which is a consequence of the Baire Category Theorem, so that nonempty open sets are non-meager. 
In the more general context of topological vector spaces, the Banach Category Theorem plays this role.
\begin{thm}[Banach Category Theorem]\label{BCT}
    In any topological space, the union of any family of meager open sets is meager.
\end{thm}
See \cite[Chapter 16]{Oxtobybook} for an exposition of this result.  As an immediate consequence, we have:
\begin{lem}
\label{bairespace}
If $V$ is a topological vector space containing a non-meager set, then any nonempty open subset of $V$ must be non-meager. 
\end{lem}
\begin{proof}
Let $V$ be such.  If there were a meager nonempty open subset $U$ of $V$, then $V$ would itself be meager by the Banach category theorem as the union of translates of $U$. But if $V$ were meager, then it could not contain a non-meager subset. 
\end{proof}

Next, we need a topological vector space analogue of the open balls of Euclidean space.  In Definition \ref{bounded}, if $\newF$ is a bounded set and we scale $\newF$ by $t_0$ to be in a given neighborhood of $\textbf{0}$, then there is no guarantee that we could then scale by a smaller $t<t_0$ and still guarantee this containment.  Fortunately, in any topological vector space, there is a topological base consisting of sets which are closed under scaling down.  This is captured in the following lemma.
\begin{dfn}
    Let $V$ be a topological vector space.  A set $B$ is called \textbf{balanced} if for any $t\leq 1$, we have $tB\subset B$.
\end{dfn}
\begin{lem}[\cite{TVS}, Theorem 4.5.1]
\label{balanced}
Let $V$ be a topological vector space.  There is a neighborhood base at $\textbf{0}$ consisting of balanced open sets.
\end{lem}

In the proof of Theorem \ref{euclidean}, the way boundedness is used is that a bounded set multiplied by a small scalar produces a set which is contained in a small ball around the origin.  We then observe that an arbitrary intersection of open sets is guaranteed to be open, provided each of the sets is a small perturbation of a fixed interval.  The following lemma will play the analogous role in this section and is a key ingredient in the proof of our main result. 
\begin{lem}
\label{intersection}
    Let $V$ be a topological vector space, let $U$ be an open set, and let $\newF$ be a bounded set.  For each $x\in \newF$, define $U_{x,t}=U-tx$.  Finally, define
    \[
    \widetilde{U}_t=\bigcap_{x\in \newF} U_{x,t}.
    \]
    There exists a non-degenerate interval $J\subset \R$ such that for any $t\in J$, the set $\widetilde{U}_t$ has nonempty interior.
\end{lem}
\begin{proof}
    First, we consider the case $\textbf{0}\in U$.  Since vector space addition $+:V\times V\to V$ is continuous, the pre-image set
    \[
    \{(x,y)\in V\times V: x+y\in U\} 
    \]
    is open.  Since Cartesian products of open sets are a basis for the product topology, there exist open sets $U_1,U_2\subset V$ such that $(\textbf{0},\textbf{0})\in U_1\times U_2$ and $U_1 + U_2\subset U$.  The set $U'=U_1\cap U_2\cap U$ is therefore an open neighborhood of $\textbf{0}$, contained in $U$, with the property that $U'+U'\subset U$.  By Lemma \ref{balanced}, we may take a subset $B\subset U'$ with these same properties, 
    which is also balanced.  Since $\newF$ is bounded, there exists $t_0$ such that $t_0\newF\subset B$.  Since $B$ is balanced, we therefore have $t\newF\subset B$ for any $t\in J:=(0,t_0)$.  We claim that for any such $t$, the set $\widetilde{U}_t$ contains $B$ as a subset.  Indeed, for any $z\in B, t\in J$, and $x\in \newF$, we have
    \[
    z+tx\in B+B\subset U,
    \]
    hence 
    \[
    z=(z+tx)-tx\in U-tx=U_{x,t}.
    \]
    Now, consider the general case.  For any open set $U$, we consider a translation $U'$ of $U$ that contains the origin.  From the first case, it follows that $\widetilde{U'}_t$ has a nonempty interior for an interval worth $t$.  Since $\widetilde{U}_t$ is a translate of $\widetilde{U'}_t$, the result follows.
\end{proof}
We are now ready to prove the main theorem of this section. 
For variety of exposition, rather than just present a modification of the proof of Theorem \ref{picard_gen} given at the beginning of this section, we present an alternative proof. 



\begin{proof}[Proof of Theorem \ref{TVS}]
    Define $A_{i,t}=A-tv^i$.  We have $\{tv^i+z:i\in\N\}\subset A$ if and only if $z\in \bigcap_{i=1}^\infty A_{i,t}$.  We must show that this intersection is nonempty for an interval worth of $t$.  Let $U$ be an open set, and let $\newP$ be a meager set such that $A=U\Delta \newP$; let $U_{i,t}=U-tv^i$ and $\newP_{i,t} = \newP-tv^i$.    Since any translation of a nowhere dense set is still nowhere dense, the set $\newP_{i,t}$ is meager for every $i,t$.  Define
    \[
    \widetilde{U}_t=\bigcap_{i=1}^\infty U_{i,t} \,\,\, \text{ and } \,\,\, \widetilde{\newP}_t=\bigcup_{i=1}^\infty  \newP_{i,t}.
    \]

We now observe that
$$  \bigcap_{i=1}^\infty A_{i,t} \supset  \widetilde{U}_t \backslash  \widetilde{\newP}_t.$$

    The proof is completed by arguing that $\widetilde{U}_t$ is non-meager; indeed, since $\widetilde{\newP}_t$ is meager, then it would follow that $ \bigcap_{i=1}^\infty A_{i,t}$ is nonempty. 
    By Lemma \ref{intersection}, there is a non-degenerate interval $J\subset \R$ such that for any $t\in J$, the set $\widetilde{U}_t$ has nonempty interior.  By Lemma \ref{bairespace}, since $V$ contains a non-meager set $A$, a set with nonempty interior must be non-meager.

\end{proof}

\appendix

\section{Examples and properties of Baire sets}
\label{Appendix}
The purpose of this appendix is to demystify the definition of a Baire set and describe some examples which show that the assumptions of our main theorems cannot be dropped.  First, we have the following alternate characterizations of Baire sets.
\begin{thm}[\cite{Kechris}, Proposition 8.23]
\label{equivalentdefinitions}
Let $X$ be any topological space, and let $A\subset X$.  The following are equivalent:
\begin{enumerate}[(i)]
\item $A$ is a Baire set.
\item $A=G\cup \newP$ for some $G_\delta$ set $G$ and some meager set $\newP$.
\item $A=F\setminus \newP$ for some $F_\sigma$ set $F$ and some meager set $\newP$.
\end{enumerate}
\end{thm}
As an immediate consequence, we have the following classes of examples of Baire sets.
\begin{cor}
\label{cor: equivalent definitions}
Let $X$ be any topological space.
\begin{enumerate}[(a)]
    \item Every meager subset of $X$ is Baire.
    \item Every $G_\delta$ or $F_\sigma$ subset of $X$ is Baire.
\end{enumerate}
\end{cor}
It turns out that these examples generate all possible Baire sets.  More precisely, we have.
\begin{thm}[\cite{Kechris}, Proposition 8.22]
In any topological space, the collection of Baire sets is the smallest $\sigma$-algebra containing all open sets and all meager sets.
\end{thm}
Since the class of Borel sets is defined as the smallest $\sigma$-algebra containing all open sets, we have:
\begin{cor}
All Borel sets are Baire.

\end{cor}
The most well-known example of a non-measurable set is also an example of a non-Baire set.  It also shows the necessity of the assumption that $A$ Baire in Theorem \ref{euclidean}.
\begin{ex}[A non-meager set $A$ for which the conclusion of Theorem \ref{euclidean} fails]
Let $\sim$ denote the equivalence relation on $\R$ given by $x\sim y$ if and only if $x-y\in\Q$.  Let $A$ be a set consisting of exactly one representative of each equivalence class.  We claim $A$ is non-meager; indeed, if $A$ were meager, then every translate of $A$ would also be meager, and thus
\[
\R=\bigcup_{q\in\Q}(A+q)
\]
would be meager (this is pointed out in \cite{Oxtobybook}, where it is used to prove directly that $A$ is not Baire).  Let $\newF=\{x,y\}\subset \R$, where $x-y\in\Q$.  By definition, for any $t\in\Q$ and $z\in\R$ we have $(tx+z)\sim (ty+z)$.  By construction, $A$ cannot contain two distinct equivalent numbers.  It follows that $\Delta_\newF(A)$ does not contain any rational number, and hence cannot have nonempty interior (in light of Theorem \ref{euclidean}, this gives an alternative proof that $A$ is not Baire).  The reader who is interested in set theory should note that the Axiom of Choice is crucial in this construction.  Indeed, there is a model of ZF without choice in which all sets are both measurable and Baire (see, for example, \cite[Theorem 26.14]{Jech})
\end{ex}

Next, we consider a simple consequence of the Baire category theorem showing that countable sets are universal in the collection of dense $G_\delta$ sets. Since every dense $G_\delta$ set is comeager, and since every comeager set contains a dense $G_\delta$ set, we see that universality in the collection of dense $G_\delta$ sets is equivalent to universality in the collection of comeager sets. 
\begin{ex}[finite and countably infinite sets are universal in the collection of comeager sets and in the collection of full measure sets ]\label{baire category ex}
If $P$ is countable and $A$ is a dense $G_\delta$ set,
then in particular, $P$ and $\R\backslash A$ are meager, and it follows that $P+\R\backslash A $ is meager.  Since $\R$ is non-meager, then $P+\R\backslash A \neq \R$, and so there is a $t\in \R$ so that $P+t\subset A$. 

A very similar argument shows that countable sets are full measure universal.  If $P$ is countable and $A$ is a set of full measure, then $P+\R\backslash A$ is countable so that $P+\R\backslash A \neq \R$, and so there is a $t\in \R$ so that $P+t\subset A$. 
\end{ex}
\vskip.12in

There is a broad literature on fractal sets which avoid given configurations, and generally, such a set must be meager.  This makes it possible to import examples which show that Theorem \ref{euclidean} (and hence the more general Theorem \ref{TVS}) fail if a set is assumed to be Baire but not necessarily non-meager.  For example, for any dimension $d$, there are well-known constructions of compact sets $A$ of dimension arbitrarily close to $d$ such that $A-A$ has empty interior (see for example \cite[Theorem 8.15]{FalconerGFS}).  In the language of this paper, this is the same as saying $\Delta_\newF(A)$ has empty interior for any set $\newF$ with at least two points.  Beyond this, we can give examples of $3$-point sets $\newF$ and Baire sets $A$ of positive Hausdorff dimension where the set $\Delta_\newF(A)$ is actually empty (rather than simply having empty interior).  To do this, we use a construction of M\'ath\'e.
\begin{thm}[\cite{mathe}]
\label{mathe}
Let $P(X_1,\dots,X_m)$ be a polynomial with rational coefficients in $dm$ variables (where each $X_i$ represents a vector variable in $\R^d$).  There exists a compact set $E\subset \R^d$ of positive Hausdorff dimension which does not contain any configuration of $m$ distinct points $x_1,\dots,x_m\in\R^d$ which lie in the zero set of $P$.
\end{thm}

The following demonstrates the necessity of the assumption that $A$ is non-meager in Theorem \ref{euclidean}. Further, it demonstrates that even if $A$ has positive Hausdorff dimension, the conclusion of this theorem need not hold. 
\begin{ex}[A meager set with positive Hausdorff dimension for which the conclusion of Theorem \ref{euclidean} fails]
Let $x,y,z\in\R^d$ be three distinct points such that the number
\[
\alpha:=\frac{|y-x|}{|z-x|}
\]
is algebraic.  It follows that $\alpha^2$ is also algebraic, so we may take 
\[
Q(X):=a_nX^n+a_{n-1}X^{n-1}+\cdots+a_1X+a_0
\]
to be the minimal polynomial of $\alpha^2$ (here, $X$ is a one-dimensional variable rather than a vector variable as in Theorem \ref{mathe}).  Plugging in our definition of $\alpha$, we have
\[
a_n\left(\frac{|y-x|^2}{|z-x|^2}\right)^n+a_{n-1}\left(\frac{|y-x|^2}{|z-x|^2}\right)^{n-1}+\cdots+a_1\left(\frac{|y-x|^2}{|z-x|^2}\right)+a_0=0.
\]
Finally, multiplying through by $|z-x|^{2n}$ to clear denominators gives us
\[
a_n|y-x|^{2n}+a_{n-1}|y-x|^{2n-2}|z-x|^2+\cdots+a_1|y-x|^2|z-x|^{2n-2}+a_0|z-x|^{2n}=0.
\]
Thus, the points of $\newF$ satisfy a polynomial equation with rational coefficients.  Moreover, by homogeneity and translation invariance, one can easily verify that the equation remains satisfied if $x,y,z$ are replaced by $tx+w,ty+w,tz+w$, respectively, for any $t>0$ and $w\in\R^d$.  By Theorem \ref{mathe}, it follows that there exists a compact set $A\subset \R^d$ such that $\Delta_\newF(A)=\emp$.  Such a set $A$ must have empty interior, so it is nowhere dense and hence meager.  However, because it is compact, it is still a Baire set by Corollary \ref{cor: equivalent definitions}.  This shows the non-meager assumption in our theorems is necessary.
\end{ex}

Our next example shows that the assumption that $\newF$ is countable in Theorem \ref{euclidean} is necessary. 

\begin{ex}[A non-meager Baire set for which Theorem \ref{euclidean} fails if countable is replaced by uncountable]\label{GLWexample}
As mentioned in the introduction, 
Gallagher, Lai, and Weber \cite[Theorem 1.4]{GLW} proved that Cantor sets in $\R^d$ are not universal in the collection of dense $G_\delta$ sets (equivalently, comeager sets). 
In particular, this implies that Cantor sets are not universal in the collection of non-meager Baire sets. 
\end{ex}
\vskip.12in

Our final (classic) example shows that although measure zero sets and meager sets can both be considered ``negligible'', these two notions of negligibility are not mutually consistent.
\begin{ex}[A meager set with full measure and a second category set with zero measure]
Enumerate the rationals by $q_1, q_2, \cdots$, 
and let $O(i,j)$ denote the open interval with center $q_j$ and length $\frac{1}{2^{i+j}}$. 
Set $U_i = \bigcup_{j=1}^\infty O(i,j)$, and set $D= \bigcap_{i=1}^\infty U_i$. One may verify that $D$ is a set of measure zero, while $M=\R \backslash D$ is meager (as the complement of an intersection of open dense sets; each $U_i$ contains $\Q$). In conclusion, $\R=M \bigcup D$ can be expressed as the disjoint union of a meager set with infinite measure and a comeager set (second category) with zero measure.  
\end{ex}



\bibliographystyle{plain} 
\bibliography{references_jan25}

\end{document}